\documentclass[11pt]{amsart}

\usepackage{comment}
\usepackage{amsfonts}
\usepackage{amsmath}
\usepackage{amssymb}
\usepackage{amsthm}
\usepackage{amsxtra}
\usepackage{epic, eepic}
\usepackage{epsfig,color}
\usepackage{fullpage}
\usepackage{graphicx}
\usepackage{subfigure}
\usepackage{vmargin}
\usepackage{enumerate}
\theoremstyle{plain}
\newtheorem*{thm:ADinBA}{Theorem \ref{thm:ADinBA}}
\newtheorem*{thm:main}{Theorem \ref{thm:main}}
\newtheorem*{thm:symmetries}{Theorem \ref{thm:symmetries}}

\newtheorem{Thm}{Theorem}[section]

\newtheorem{Lem}[Thm]{Lemma}
\newtheorem{Prop}[Thm]{Proposition}

\theoremstyle{definition}

\newtheorem*{Ques}{Question}

\newtheorem{con}{Conjecture}
\theoremstyle{remark}
\newtheorem{Rem}{Remark}[section]

\newtheorem{Ex}{Example}[section]

\definecolor{violet}{RGB}{204,52,255}

\numberwithin{equation}{section}

\setmarginsrb{33mm}{30mm}{33mm}{40mm}%
            {0mm}{10mm}{0mm}{10mm}

\begin{document}

\newcommand{\rx}{\mathcal{R}(X)}

% Title
\title{Perfect powers in Catalan and Narayana numbers}

\author{Sara Checcoli$^{*}$} \thanks{$^{*}$ supported by an SFSN grant}
\address{Universit\"at Basel \\Mathematisches Institut \\ Rheinsprung 21\\
CH-4051 Basel, Switzerland}\email{sara.checcoli@gmail.com}
\author{Michele
D'Adderio}
\address{Universit\'e Libre de Bruxelles (ULB)\\D\'epartement de Math\'ematique\\
Boulevard du Triomphe, B-1050 Bruxelles\\
Belgium}\email{mdadderi@ulb.ac.be}

\begin{abstract}
When a Catalan number or a Narayana number is a (non-trivial) perfect power?
For Catalan numbers, we show that the answer is ``never''. However, we prove that for every $b$,
the Narayana number $N(a,b)$ is a (non-trivial) perfect square for infinitely many
values of $a$, and we show how to compute all of them. We also
conjecture that $N(a,b)$ is never a (non-trivial) perfect $k$-th
power for $k\geq 3$ and we prove some cases of this conjecture.
\end{abstract}

\maketitle
%\tableofcontents

\section*{Introduction}

Given two natural numbers $a$ and $b$, the \emph{Narayana number}
$N(a,b)$ is defined by the formula
$$
N(a,b):=\frac{1}{a}\left(\!\!\!%
\begin{array}{c}
  a \\
  b \\
\end{array}%
\!\!\!\right)\left(\!\!\!%
\begin{array}{c}
  a \\
  b-1 \\
\end{array}%
\!\!\!\right).
$$
These numbers are well known in discrete mathematics, since they
count several families of mathematical objects (see
\cite{delestviennot} for a classical reference, or
\cite{aigner,fominreading,hivertnovellithibon,huq,mansoursun,mathews,sulanke,williams,yanoyoshida}
for some more recent occurrence), e.g. $N(m+n-1,m)$ is the number
of parallelogram polyominoes in a rectangular $m\times n$ box (cf.
\cite{avaldadderiodukesleborgne}).

There is a natural link with the famous \emph{Catalan numbers}
$$
C_n:=\frac{1}{n+1}\left(\!\!\!\begin{array}{c}
  2n \\
  n \\
\end{array}\!\!\!\right),
$$
given by the identity $\sum_kN(n,k)=C_n$.

The Catalan numbers are ubiquitous in mathematics: see
\cite[Exercise 6.19]{stanleybook,stanleyaddendum} for about 200
families of mathematical objects counted by these numbers.

In algebraic combinatorics $q,t$-analogues of Catalan numbers have
been studied in connection with the so called $n!$-conjecture (now
$n!$-theorem of Haiman) about the renown diagonal harmonics (see
\cite{GHCat,haglundbook,haiman}). More recently, a $q,t$-analogue
of the Narayana numbers has been shown to be intimately related to
the decade old shuffle conjecture about the Frobenius
characteristic of the diagonal harmonics (see
\cite{avalbergerongarsia,addhl,haglund}), renewing the interest
for these numbers.

From a number theoretic point of view, it is natural to ask about
divisibility properties of these numbers, which are of course
related to divisibility properties of the binomial coefficients.

For the Catalan numbers, such properties have been studied by
several authors. In particular their parity was studied in
\cite{AC}, while more generally their congruence modulo a power of
$2$ has been recently investigated in \cite{liuyeh,xinxu}. Their
divisibility by prime powers was completely determined in
\cite{AltKub} via arithmetic techniques. In \cite{DS}, among other
results, the $2$-adic valuation of Catalan numbers has been
studied by means of certain group actions. Similar results for
generalizations of Catalan numbers have been studied in
\cite{konvalinka,postnikovsagan}.

Lately, also the Narayana numbers have received more attention in
this direction. In particular, in \cite{bonasagan}, using a
theorem of Kummer on the $p$-adic valuation of binomial
coefficients,  the authors study the divisibility of $N(a,b)$ by
primes, in relation with the description of $N(a,b)$ in base
$p$.\newline

In this work we study the following number theoretic question:
\begin{Ques}
When a Catalan number or a Narayana number is a (non-trivial) perfect power?
\end{Ques}
For us an integer is a \emph{(non-trivial) perfect power} if it is
of the form $m^k$ where $m$ and $k$ are both integers $\geq 2$ (so
$1$ is \emph{not} a perfect power). For $k=2$, we call this
integer a \emph{perfect square}.
\smallskip

The question for Catalan numbers has a negative answer: it
follows easily from a classical theorem of Ramanujan on the
distribution of primes in intervals that the sequence of Catalan
numbers does not contain perfect powers. We show this in Section
\ref{cat}.

Interestingly enough, the situation for the Narayana numbers is
very different: there are a lot of them which are perfect squares.

We study this case in Section \ref{sq}. We start by exhibiting
infinitely many pairs $(a,b)$ such that $N(a,b)$ is a perfect
square, see Proposition \ref{quad}. Then, in Theorem \ref{main}, we
give an effective algorithm to compute all such pairs, proving in
particular the stronger result that for any given $b>1$, there are
infinitely many $a>b$ such that $N(a,b)$ is a perfect square. It
turns out that this problem can be reduced to the study of certain
generalized Pell's equations: we recall the facts about these
equations that we need in Section \ref{gPell}.

In Section \ref{computations} we show how
the algorithm works with an explicit example.

We conclude, in Section \ref{hp}, by studying the seemingly more
complicated case of higher powers. We make the following
conjecture.
\begin{con}\label{conNara} $N(a,b)$ is never a non-trivial perfect $k$-th power for $k\geq 3$.
\end{con}

Conjecture \ref{conNara} seems to be quite hard. For example, for
$b=3$ (the case $b=2$ follows from some known results), it is
related to a generalization of the Catalan's conjecture by Pillai.
However we are able to provide some evidence for our conjecture,
by showing that it holds when $b$ is ``not too small'' (see
Theorems \ref{thm1main} and \ref{thm2main}).

\section{Catalan numbers are not perfect powers}\label{cat}

The famous \emph{Catalan numbers} are defined, for $n\geq 1$, by
the formula
$$
C_n:=\frac{1}{n+1}\left(\!\!\!\begin{array}{c}
  2n \\
  n \\
\end{array}\!\!\!\right).
$$

In this section we answer the question:
\begin{Ques}
Are there perfect powers in the Catalan sequence $\{C_n\}_{n\geq
1}$?
\end{Ques}
Though Catalan numbers have been extensively studied, to the best
of our knowledge this is the first investigation of this kind.

The negative answer to our question follows easily from the
following classical theorem, which is due to Ramanujan (see
\cite[Section 9.3B]{shap} for a proof).
\begin{Thm}[Ramanujan]
For $n\geq 6$ there are at least two primes between $n$ and $2n$.
\end{Thm}

Here is a complete answer to our question.
\begin{Thm}
For all $n$, the $n$-th Catalan number $C_n$ is never a perfect
power.
\end{Thm}
\begin{proof}
We write \begin{equation}\label{eq:catalan}
C_n=\frac{1}{n+1}\left(\!\!\!\begin{array}{c}
  2n \\
  n \\
\end{array}\!\!\!\right)=\frac{(2n)(2n-1)\cdots(n+2)}{n!}.
\end{equation}
By Ramanujan's Theorem, for $n\geq 6$ there are at least two
primes between $n$ and $2n$. Since they cannot be both $n$ and
$n+1$, this implies that there is at least one prime between $n+2$
and $2n$. Hence this prime divides exactly $C_n$ (since it divides
the numerator in \eqref{eq:catalan}, but not the denominator),
showing that it cannot be a perfect power.

Since $C_1=1, C_2=2, C_3=5, C_4=14$ and $C_5=42$ are not perfect
powers, this completes the proof.
\end{proof}

\section{The generalized Pell's equation $n^2-dm^2=z^2$} \label{gPell}

Before considering the problem of when a Narayana number is a
non-trivial perfect power, we recall some results on Pell's
equation which will be used in the following section.\newline

Let $z$ and $d$ be positive integers, with $d$ squarefree. In this
section we want to describe all the positive integral solutions
$(n,m)$ \textbf{with $m$ even} of the \emph{generalized Pell's
equation}
\begin{equation}\label{pell1}
n^2-dm^2=z^2.
\end{equation}

We start recalling the following classical well known results. We
suggest \cite[Chapter 1]{Zan} as general reference.

\subsection{Solving the generalized Pell's equations
$n^2-dm^2=z^2$}

From \cite[Proposition 1.5]{Zan}, all positive integral solutions
$(n,m)$ of the \emph{generalized Pell's equation}
\begin{equation}\label{pell1}
n^2-dm^2=z^2
\end{equation}
are obtained as
\[n\pm m\sqrt{d}=(n'+m'\sqrt{d})(n_1+m_1\sqrt{d})^k\] where $k\in \mathbb{Z}$,
$(n_1,m_1)$ is the \emph{fundamental solution} of the
corresponding \emph{Pell's equation}
\begin{equation}\label{pell2}
n^2-dm^2=1,
\end{equation}
and $(n',m')$ is a particular positive solution of \eqref{pell1},
belonging to a finite set effectively computable only in terms of
$n_1,m_1,d$ and $z$. In particular $(n',m')$ can be chosen so that
\[|n'|<z\sqrt{n_1+m_1\sqrt{d}} \phantom{rrrr} \text{and}
\phantom{rrrr}  |m'|<z\sqrt{\frac{n_1+m_1\sqrt{d}}{d}}.\]

The fundamental solution $(n_1,m_1)$ of \eqref{pell2} is easily
computable: $n_1/m_1$ is in fact the truncation of the continued
fraction expansion of $\sqrt{d}$ to the end of its first period,
if this period has even length, or to the end of its second
period, if this period has odd length. Moreover, all positive
solutions of \eqref{pell2} are of the form $(n_k,m_k)$ with
$$
n_k+m_k\sqrt{d}:=(n_1+m_1\sqrt{d})^k \quad \text{ for }k\in
\mathbb{N}.
$$
\subsection{Finding integral solutions $(n,m)$ with $m$ even}
We go back to the original purpose of this section, that is to find integral solutions of \eqref{pell1} with $m$ even.

Observe that if $d$ is even, then $n$ and $z$ must have the same
parity. We have
$$
dm^2=n^2-z^2=(n-z)(n+z),
$$
hence in this case $dm^2$ is divisible by $2^2=4$; since $d$ is
squarefree, we must have that $2$ divides $m^2$, so $m$ is even.

\smallskip

Therefore, we assume from now on that $d$ is odd.

\smallskip

Consider the general product
$$
x_2+y_2\sqrt{d}:=(x_1+y_1\sqrt{d})(x_0+y_0\sqrt{d}),
$$
where all $x_i$'s and $y_i$'s are integers. We have
$$
(x_1+y_1\sqrt{d})(x_0+y_0\sqrt{d})=(x_1x_0+y_1y_0 d)+\sqrt{d}
(x_1y_0+x_0y_1),
$$
hence
$$
x_2=x_1x_0+y_1y_0 d\quad \text{ and } \quad y_2=x_1y_0+x_0y_1.
$$
It is now clear that:
\begin{itemize}
\item if $x_0$ and $x_1$ are both even, and $y_0$ and $y_1$ are
both odd, then $x_2$ is odd while $y_2$ is even; \item similarly,
if $y_0$ and $y_1$ are both even, and $x_0$ and $x_1$ are both
odd, then $x_2$ is odd while $y_2$ is even.
\end{itemize}
On the other hand
\begin{itemize}
\item if $x_0$ and $y_1$ are both even, and $y_0$ and $x_1$ are
both odd, then $x_2$ is even while $y_2$ is odd.
\end{itemize}

Moreover
\begin{itemize}
\item if only one of $x_0,y_0,x_1,y_1$ is odd, then both
$x_2$ and $y_2$ are even; \item if only one of $x_0,y_0,x_1,y_1$
is even, then both $x_2$ and $y_2$ are odd.
\end{itemize}
\smallskip

Now, since $(n_k, m_k)$ are all solutions of the Pell's equation
\eqref{pell2}, necessarily $n_k$ and $m_k$ have different parities
for all $k$ (we are assuming that $d$ is odd!).

In fact, from what we observed, we easily deduce that if $n_1$ is
odd, then $n_k$ is odd (and hence $m_k$ is even) for all $k\in
\mathbb{Z}$. Similarly, if $n_1$ is even, then $n_{2k+1}$ is even
(and hence $m_{2k+1}$ is odd) for all $k\in \mathbb{Z}$, while
$n_{2k}$ is odd (and hence $m_{2k}$ is even) for all $k\in
\mathbb{Z}$.

\smallskip

Now all positive solutions $(n,m)$ of \eqref{pell1} are obtained
as
$$
n\pm m\sqrt{d}=(n'+m'\sqrt{d})(n_1+m_1\sqrt{d})^k=
(n'+m'\sqrt{d})(n_k+m_k\sqrt{d})\quad \text{ for }k\in \mathbb{Z},
$$
from which $m=|n'm_k+m'n_k|$. So the parity of $m$ depends on the
parities of $n'$ and $m'$. From the discussion above, we easily
deduce the following lemma.
\begin{Lem} \label{lemsol}
In the notation above, all positive solutions $(n,m)$ of
\eqref{pell1} with $m$ even are obtained as
$$
n\pm m\sqrt{d}=(n'+m'\sqrt{d})(n_k+m_k\sqrt{d})\quad \text{ for
}k\in \mathbb{Z},
$$ if $d$ is even or if both $n'$ and $m'$ are even;
$$
n\pm m\sqrt{d}= (n'+m'\sqrt{d})(n_{2k}+m_{2k}\sqrt{d})\quad \text{
for }k\in \mathbb{Z},
$$
if both $n'$ and $d$ are odd, while $m'$ is even;
$$
n\pm m\sqrt{d}= (n'+m'\sqrt{d})(n_{2k+1}+m_{2k+1}\sqrt{d})\quad
\text{ for }k\in \mathbb{Z},
$$
if both $m'$ and $d$ are odd, while $n'$ is even.
\end{Lem}

\section{The case of the squares}\label{sq}
Given two natural numbers $a$ and $b$, the \emph{Narayana number}
$N(a,b)$ is defined by the formula
$$
N(a,b):=\frac{1}{a}\left(\!\!\!%
\begin{array}{c}
  a \\
  b \\
\end{array}%
\!\!\!\right)\left(\!\!\!%
\begin{array}{c}
  a \\
  b-1 \\
\end{array}%
\!\!\!\right).
$$

In this section we study when $N(a,b)$ is a perfect square.

Since given two natural numbers $a$ and $b$, we have
$N(a,a)=N(a,1)=1$, while $N(a,b)=0$ for $a<b$, we will always
assume in what follows that $a> b> 1$.

\medskip

\begin{comment}
We observe that
$$
N(a,b)=\frac{1}{a}\left(\!\!\!%
\begin{array}{c}
  a \\
  b \\
\end{array}%
\!\!\!\right)\left(\!\!\!%
\begin{array}{c}
  a \\
  b-1 \\
\end{array}%
\!\!\!\right)=\frac{b}{a(a-b+1)}\left(\!\!\!%
\begin{array}{c}
  a \\
  b \\
\end{array}%
\!\!\!\right)^2=\frac{(a-b+1)}{ab}\left(\!\!\!%
\begin{array}{c}
  a \\
  b-1 \\
\end{array}%
\!\!\!\right)^2.
$$
\end{comment}

\medskip

It is not hard to see that there are infinitely many pairs $(a,b)$
for which $N(a,b)$ is a square. We can prove a little more by
providing explicit families of such pairs.

\begin{Prop}\label{quad}
There are infinitely many pairs $(a,b)$ such that $N(a,b)$ is a
square.
More precisely:
\begin{enumerate}
\item\label{s1}  if $n$ is odd, then
$N\left(n^2,\frac{n^2+1}{2}\right)$ is a square; \item\label{s2}
if $n$ is even, then  $N\left(n^2-2,\frac{n^2-2}{2}\right)$ is a
square; \item\label{s3} for all $n$, $N(n^2(n^2+1), n^2+1)$ is a
square.
\end{enumerate}
\end{Prop}
\begin{proof}
We start with the following simple, but quite useful,
manipulation:
\begin{equation}\label{eq1}
N(a,b)=\frac{1}{a}\left(\!\!\!%
\begin{array}{c}
  a \\
  b \\
\end{array}%
\!\!\!\right)\left(\!\!\!%
\begin{array}{c}
  a \\
  b-1 \\
\end{array}%
\!\!\!\right)=\frac{b}{a(a-b+1)}\left(\!\!\!%
\begin{array}{c}
  a \\
  b \\
\end{array}%
\!\!\!\right)^2.
\end{equation}

Hence to check that $N(a,b)$ is a square it is enough to check
that $b/a(a-b+1)$ is.

For $n\in \mathbb{N}$, $n\geq 1$ odd, $(n^2+1)/2$ is a positive
integer, so, letting $a:=n^2$ and $b:=(n^2+1)/2$ we compute
$$
\frac{a(a-b+1)}{b}=\frac{2}{(n^2+1)}\left(n^2\left(n^2-\frac{(n^2+1)}{2}+1\right)\right)=n^2.
$$
This shows that, for $n>1$ odd $N\left(n^2,\frac{n^2+1}{2}\right)$
is always a square, proving  \eqref{s1}.

Similarly, for $n\in \mathbb{N}$, $n>1$ even, $(n^2-2)/2$ is a
positive integer. Setting $a:=n^2-2$ and $b:=(n^2-2)/2$, we
compute
$$
\frac{a(a-b+1)}{b}=\frac{2}{(n^2-2)}\left((n^2-2)\left((n^2-2)-
\frac{(n^2-2)}{2}+1\right)\right)=
2\left(\frac{(n^2-2)}{2}+1\right)=n^2.
$$
So for $n>2$ even, $N\left(n^2-2,\frac{n^2-2}{2}\right)$ is always
a square, establishing
\eqref{s2}.

Finally, for any integer $n\in \mathbb{N}$, letting
$a:=n^2(n^2+1)$ and $b:= n^2+1$ gives
$$
\frac{a(a-b+1)}{b}=\frac{n^2(n^2+1)\left(n^2(n^2+1)-
(n^2+1)+1\right)}{(n^2+1)}=n^6,
$$
so $N\left(n^2(n^2+1),n^2+1\right)$ is a square too, proving \eqref{s3}.
\end{proof}
\begin{Rem}
Notice that Proposition \ref{quad} does not cover all the pairs
$(a,b)$ such that $N(a,b)$ is a square. For instance $N(1728,28)$
and $N(63,28)$ are both squares (as we will see in the next
section) but they are not in the families appearing in Proposition
\ref{quad}.
\end{Rem}
\medskip

In fact we can do much better: for any given $b$, we can produce
all the $a$'s for which $N(a,b)$ is a square. It turns out that
there are infinitely many of them for every $b$.

We will show that the problem of finding all pairs $(a,b)$ such
that $N(a,b)$ is a square reduces to finding solutions of a
general Pell's equation.
\smallskip

The following theorem is the main result of this section. We
remark here that its proof gives an algorithm to compute all the
pairs $(a,b)$ for which $N(a,b)$ is a perfect square. Some explicit computations will be made in Section \ref{computations}.

\begin{Thm}\label{main}
For every fixed integer $b>1$, $N(a,b)$ is a perfect square for
infinitely many and  effectively computable integers $a$.
\end{Thm}
\begin{proof}
Let $b$ be a positive integer, with $b=ds^2$ and $d$ is
square-free.

We want to find all integers $a$'s such that  $N(a,b)=c^2$ for
some integer $c$. Using \eqref{eq1}, this is equivalent to
$$
N(a,b)=\frac{b}{a(a-b+1)}\left(\!\!\!%
\begin{array}{c}
  a \\
  b \\
\end{array}%
\!\!\!\right)^2=c^2,
$$
so $N(a,b)$
is a square if and only if $ab(a-b+1)$ is a square.

We now show that this problem is equivalent to finding the
integral solutions $(n,m)$ of the Pell's equation
\begin{equation} \label{eq:genPell}
n^2-dm^2=(b-1)^2
\end{equation}
such that $m$ is even.

In fact, assume that $a$ is an integer such that
\begin{equation}\label{equs}ab(a-b+1)={c'}^2\end{equation} for
some integer $c'$. Notice that from \eqref{equs}, $b$ divides
$(c')^2$, hence $ds$ divides $c'$. It is now easy to check that
the pair
\[(n,m)=\left(2a+1-b, 2s\frac{c'}{b}\right)\] is an integral
solution of \eqref{eq:genPell} with $m$ even.

On the other hand, suppose to be given a solution $(n,m)$ of
\eqref{eq:genPell} with $m$ even. Notice that this implies that
$n$ and $b-1$ have the same parity. So \[a=\frac{n+b-1}{2}\] is an
integer and one can easily check  that
$$
N(a,b)=\left(\frac{2s}{m} \binom{a}{b}\right)^2.
$$

Now Lemma \ref{lemsol} shows how to compute all the positive
solutions $(n,m)$ of equation \eqref{eq:genPell} with $m$ even.

There are always infinitely many, since for example, in the
notation of the lemma, we can always choose $(n',m')=(b-1,0)$.
This completes the proof.
\end{proof}

\section{Some explicit computation} \label{computations}

We show how the proof of Theorem \ref{main} is effective by
computing an explicit example.

Let $b=28$. We want all the $a$'s such that $N(a,28)$ is a square
greater than $1$. The algorithm is the following.

Keeping the above notation, we write $b=7\cdot 2^2$, so $d=7$,
$s=2$ and equation \eqref{eq:genPell} becomes
\begin{equation}\label{eq7}
n^2-7m^2={27}^2.
\end{equation}
To solve it, from the discussion of Section \ref{gPell}, first we
have to find the fundamental solution $(n_1,m_1)$ of the equation
\begin{equation}\label{equa8}
n^2-7m^2=1.
\end{equation}
The continued fraction expansion of $\sqrt{7}$ is
\begin{equation}\label{contfr}
\sqrt{7}=2+\cfrac{1}{1+\cfrac{1}{1+\cfrac{1}{1+\cfrac{1}{4+\ldots}}}},
\end{equation}
or better, in the standard notation,
$\sqrt{7}=[2,1,1,1,4,1,1,1,4,\dots]$. So it has period of length
$4$, which is even, hence truncating the expansion in
\eqref{contfr} at the end of the first period we get
$\frac{8}{3}$. Therefore
 $(n_1,m_1)=(8,3)$ is the fundamental solution we sought for.

Now all positive solutions $(n,m)$ of \eqref{eq7} can be found as
$$
n\pm \sqrt{7} m=(n'+m'\sqrt{7})(8+3\sqrt{7})^k,
$$
where $k\in \mathbb{Z}$ and $(n',m')$ is any solution of
\eqref{eq7} such that
\begin{equation}\label{range}
|n'|<27\sqrt{8+3\sqrt{7}}<27\cdot 4 \phantom{rrrr} \text{and}
\phantom{rrrr} |m'|<27\sqrt{\frac{8+3\sqrt{7}}{7}}<27\cdot 2.
\end{equation}
So it sufficient to compute $(n')^2-7(m')^2$ for all such values
of $n'$ and $m'$ and see which one satisfies \eqref{eq7}. In this
case all the solutions $(n',m')$ of \eqref{eq7} in the range
\eqref{range} are
\begin{equation}\label{list}
(27, 0),\,\, (29, 4),\,\, (36, 9),\,\, (48, 15)\,\, (69,24),\,\,
(99,36).
\end{equation}

We denote by $(n_k,m_k)$ the solution of \eqref{equa8} obtained as
$$
n_k+m_k\sqrt{7}=(8+3\sqrt{7})^k
$$
with $k\in \mathbb{Z}$.

All the positive solutions $(n,m)$ of \eqref{eq7} are then of the
form
$$
n\pm m\sqrt{d}=(n'+m'\sqrt{d})(n_k+m_k\sqrt{d})\quad \text{ for
}k\in \mathbb{Z}
$$
where $(n',m')$ is in our list \eqref{list}.
\begin{Rem}
In fact notice that
$$
(36+9\sqrt{7})(n_{-1}+m_{-1}\sqrt{7})=(36+9\sqrt{7})(8-3\sqrt{7})=99-36\sqrt{7},
$$
and consequently
$$
(99+36\sqrt{7})(n_{-1}+m_{-1}\sqrt{7})=(99+36\sqrt{7})(8-3\sqrt{7})=36-9\sqrt{7}.
$$
Similarly,
$$
(48+15\sqrt{7})(n_{-1}+m_{-1}\sqrt{7})=(48+15\sqrt{7})(8-3\sqrt{7})=69-24\sqrt{7},
$$
and consequently
$$
(69+24\sqrt{7})(n_{-1}+m_{-1}\sqrt{7})=(69+24\sqrt{7})(8-3\sqrt{7})=48-15\sqrt{7}.
$$

So we can restrict our list \eqref{list} to
\begin{equation}\label{newlist}
(27, 0),\,\, (29, 4),\,\, (36, 9),\,\, (48, 15).
\end{equation}
\end{Rem}

Remember that we are looking for integral solutions of \eqref{eq7}
with $m$ even.

Since $d=7$ is odd and $n_1=8$ is even, applying Lemma
\ref{lemsol} we have that all the positive solutions $(n,m)$ of
\eqref{eq7} with $m$ even are of the form
$$
n\pm m\sqrt{7}=\left\{\begin{array}{l}
  27(n_{2k}+m_{2k}\sqrt{7}) \\
  (29+4\sqrt{7})(n_{2k}+m_{2k}\sqrt{7}) \\
  (36+9\sqrt{7})(n_{2k+1}+m_{2k+1}\sqrt{7}) \\
  (48+15\sqrt{7})(n_{2k+1}+m_{2k+1}\sqrt{7}) \\
\end{array}\right. \quad \text{ with }k\in \mathbb{Z}.
$$

Now for all such solutions $(n,m)$, the proof of Theorem
\ref{main} shows that
$$
a:=\frac{n+b-1}{2}=\frac{n+27}{2}
$$
is an integer such that $N(a,b)=N(a,28)$ is a perfect square
whenever $a>28=b$.

\smallskip

In this way we can effectively construct, by a finite search in an
explicit bounded interval, all the pairs $(a,28)$ for which
$N(a,28)$ is a square.
 \medskip

We now list some explicit computations.
 \begin{Ex}

Let us take $(n',m')=(27,0)$.

\smallskip

For $k=0$, $(n_0,m_0)=(1,0)$ and we get $a=(1+27)/2=14<b$, which
we disregard.

\smallskip

For $k=1$ we get $(n_{2k},m_{2k})=(n_2,m_2)=(127,48)$, so
$(n,m)=(27\cdot 127,27\cdot 48)=(3429,1296)$ and
$$
a=\frac{3429+ 27}{2}=1728
$$
for which
\begin{eqnarray*}
N(1728,28) & = &
36393925811128600489003879513323005869574641433293468096956^2\\
 & = & \left(\frac{2\cdot 2
}{1296}\binom{1728}{28}\right)^2
\end{eqnarray*}

\smallskip

For $k=2$ we have $(n_{2k},m_{2k})=(n_4,m_4)=(32257,12192)$ and
$(n,m)=(27\cdot 32257 ,27\cdot 12192)=(870939,329184)$, which
gives $a=(870939+27)/2=435483$. Indeed it is possible to check
that
$$
N(435483,28)=\left(\frac{2\cdot 2
}{329184}\binom{435483}{28}\right)^2.
$$

\smallskip
\end{Ex}
\begin{Ex}
As another example, take $(n',m')=(36,9)$ from the list.

\smallskip

For $k=-1$ we have
\begin{eqnarray*}
(36+9\sqrt{7})(n_{2k+1}+m_{2k+1}\sqrt{7}) & = &
(36+9\sqrt{7})(n_{-1}+m_{-1}\sqrt{7})\\
 & = & (36+9\sqrt{7})(8-3\sqrt{7})=99-36\sqrt{7}.
\end{eqnarray*}
So $(n,m)=(99,36)$, hence $a=(99+27)/2=63$ and indeed
$$
N(63,28)=69923143311577493^2=\left(\frac{2\cdot 2
}{36}\binom{63}{28}\right)^2.
$$

\smallskip

Finally, for $k=0$ we have
\begin{eqnarray*}
(36+9\sqrt{7})(n_{2k+1}+m_{2k+1}\sqrt{7}) & = &
(36+9\sqrt{7})(n_{1}+m_{1}\sqrt{7})\\
 & = & (36+9\sqrt{7})(8+3\sqrt{7})=477+180\sqrt{7},
\end{eqnarray*}
which gives $(n,m)=(477,180)$. Therefore $a=(477+27)/2=252$, and
$$
N(252,28)=266280675495914347757098255444196475^2=\left(\frac{2\cdot
2}{180}\binom{252}{28}\right)^2.
$$
\end{Ex}

\medskip

It is amusing to see how the pairs $(a,b)$ for which $N(a,b)$ is a
square distribute. We plotted in Figure 1 such pairs for $a\leq
2000$, but only for the values $b\leq a/2$, because of the
symmetry $N(a,b)=N(a,a-b+1)$.

\vfill

\begin{figure}[h]
\includegraphics[scale=0.35]{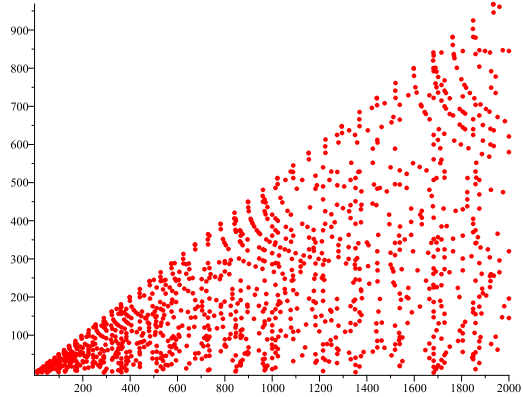}%[width=260mm,clip=true,trim=10mm 180mm 30mm 10mm]{Figure1.eps}
\caption{The red dots are the pairs $(a,b)$ with $b\leq a/2$ such
that $N(a,b)$ is a square.}
\end{figure}

\section{Higher powers}\label{hp}

In this section we investigate when a Narayana number is a perfect
power $m^k$ of some integer $m$ with $k>2$. Compared to the case
of squares, here things become more complicated.

We are concerned here with Conjecture \ref{conNara} from the
introduction, i.e. that no Narayana number is a perfect $k$-th
power of an integer for $k\geq 3$.

While the full conjecture seems to be out of reach, we provide
here evidences by presenting some partial results.

Consider the equation
$$
N(a,b)=m^k\quad \text{ for integers }\,\, m\geq 2\,\, \text{ and
}\,\, k\geq 1.
$$
From this and \eqref{eq1} we get the two equations
\begin{equation} \label{eq:1}
(a-b+1)\left(\!\!\!%
\begin{array}{c}
  a \\
  b-1 \\
\end{array}%
\!\!\!\right)^2=ab\, m^k
\end{equation}
and
\begin{equation} \label{eq:2}
b\left(\!\!\!%
\begin{array}{c}
  a \\
  b \\
\end{array}%
\!\!\!\right)^2=a(a-b+1)m^k.
\end{equation}

We start with the following proposition.
\begin{Prop} \label{lem:a=p}
Let $a,b$ be positive integers with $b\leq a/2$. Suppose that $N(a,b)=m^k$ for some positive integers $m$ and $k$.
Then:
\begin{enumerate}
\item\label{casop} if $a=p$ is a prime, then $k=1$;
\item\label{casop2} if $a=p^2$ is the square of a prime, then $k\leq 2$.
\end{enumerate}
\end{Prop}
\begin{proof}
Case \eqref{casop}: Observe that clearly $p$ does not divide both
$b$ and $p-b+1$;
moreover $p$ divides $\left(\!\!\!%
\begin{array}{c}
  p \\
  b \\
\end{array}%
\!\!\!\right)$ exactly once. Hence \eqref{eq:2} with $a=p$ implies
that $p$ divides $m^k$ exactly once, therefore we must have $k=1$.
\smallskip\\
Case \eqref{casop2}: We start by recalling the following formula.
For a prime $p$, we denote by $v_p$ the \emph{$p$-adic valuation},
i.e. for $n\in \mathbb{N}$, $v_p(n)$ is the greatest nonnegative
integer $h$ such that $p^h$ divides $n$. Then it is well known and
easy to show that
\[v_p\binom{p^n}{t}=n-v_p(t)\] for all positive integers $n,t$.

So \[v_p(N(p^2,b))=-2+(2-v_p(b))+(2-v_p(b-1)).\] Notice that $p$
cannot divide both $b$ and $b-1$. Moreover $v_p(b)$ and $v_p(b-1)$
are either 0 or 1, since $a=p^2>b>b-1$. Summing up, this tells us
that $v_p(N(p^2,b))\in \{1,2\}$ and so $N(p^2,b)$ cannot be a
perfect $k$-th power with $k>2$.
\end{proof}

\begin{Rem}
Notice that for $a=p^r$ with $r\geq 3$, and for general $b$, the
same argument only shows that $k\leq r$. So in this case we need
another strategy.
\end{Rem}

We are going to use the following result, well known as Bertrand's
postulate, and first proved by Tchebyshev.
\begin{Thm}[Tchebyshev]
For all $n>0$ there is a prime $p$ such that $n<p\leq 2n$.
\end{Thm}

The following theorem is the main result of this section.

\begin{Thm} \label{thm1main}
Let $a$ be a positive integer and let $p$ be the biggest prime
such that $p<a$. Suppose that $a/2\geq b>a-p+1$. Then $N(a,b)=m^k$
for some $m\in \mathbb{N}$, only for $k\leq 2$.
\end{Thm}
\begin{proof}
By  Tchebyshev's Theorem, we can find a prime $p$ such that
$\lfloor (a+1)/2\rfloor <p <a+1$. Because of Lemma \ref{lem:a=p},
we can assume that $p<a$. Observe that $p$ cannot divide $a$.

By assumption $a/2\geq b$, so $p>b$. Hence $p$ does not divide
$b$.

We look first at the case where $p\neq a-b+1$.

We set $c:=a-p$, so that $a-b+1=p+c-b+1$. Now $p$ does not divide
$a-b+1$, since it does not divide $c-b+1$: indeed $b-c-1>0$, thus
$b-c-1\leq b-2<p$. Now \eqref{eq:2} with $a=p+c$ becomes
$$
b\left(\!\!\!%
\begin{array}{c}
  p+c \\
  b \\
\end{array}%
\!\!\!\right)^2=(p+c)(p+c-b+1)m^k.
$$
By hypothesis $b>c+1$, hence $p$ divides
$\left(\!\!\!\begin{array}{c}
  p+c \\
  b \\
\end{array}%
\!\!\!\right)$ exactly once. Since $p$ does not divide $b$,
$p+c=a$ and $p+c-b+1=a-b+1$, it must divide $m$. But $p$ divides
the left hand side exactly twice, so it must divide $m^k$ exactly
twice. In particular we must have $k\leq 2$.

It remains to check the case where $p$ is equal to $a-b+1$. If we
set $a-b+1=p$ in the equation \eqref{eq:1}, we get

$$
p\left(\!\!\!%
\begin{array}{c}
  a \\
  p \\
\end{array}%
\!\!\!\right)^2=ab\, m^k.
$$
Now, since $\lfloor (a+1)/2\rfloor <p<a$, $p$ does not divide $\left(\!\!\!%
\begin{array}{c}
  a \\
  p \\
\end{array}%
\!\!\!\right)$. Since $p$ does not divide both $a$ and $b$, it
must divide exactly once $m^k$, which implies $k=1$.

This completes the proof of the theorem.
\end{proof}

For $c\in \mathbb{N}$, let us call $P(c)$ the greatest prime $p$
that divides $c$. It can be shown \cite{shoreytijdeman} (see also
\cite{laishramshorey}) that for $n\geq 2k>0$
$$
P\left(\left(\!\!\!%
\begin{array}{c}
  n \\
  k \\
\end{array}%
\!\!\!\right)\right)>1.95 k.
$$
\begin{Thm} \label{thm2main}
If $N(a,b)=m^k$ and $a/2\geq b\geq \sqrt{a}/1.95$, then $k\leq 2$.
\end{Thm}
\begin{proof}
Using \eqref{eq:2}, we can rewrite the condition $N(a,b)=m^k$ as
\begin{equation} \label{eq:3}
b\left(\!\!\!%
\begin{array}{c}
  a \\
  b \\
\end{array}%
\!\!\!\right)^2=a(a-b+1)m^k.
\end{equation}
Let
$$
p:=P\left(\left(\!\!\!%
\begin{array}{c}
  a \\
  b \\
\end{array}%
\!\!\!\right)\right).
$$
From what we observed before this theorem, we know that
$p>1.95b\geq \sqrt{a}$.

By a theorem of Mignotte \cite{mignotte} (see also \cite{selmer}),
if we have the prime factorization
$$
\left(\!\!\!%
\begin{array}{c}
  n \\
  k \\
\end{array}%
\!\!\!\right)=\frac{n(n-1)\cdots
(n-k+1)}{k!}=p_{1}^{\alpha_1}p_{2}^{\alpha_2}\cdots
p_{j}^{\alpha_j},
$$
then each prime power $p_{i}^{\alpha_i}$ must divide one of the
factors of the numerator.

Notice that $p^2>a$, so, using Mignotte's theorem, $p$ divides $\left(\!\!\!%
\begin{array}{c}
  a \\
  b \\
\end{array}%
\!\!\!\right)$ exactly once.

Now $p$ does not divide $b$ since $p>1.95b>b$, so $p$ divides the
left hand side of \eqref{eq:3} exactly twice.

But $p$ cannot divide both $a$ and $a-b+1$, so it divides $m^k$
one or two times. This implies that $k\leq 2$, as we wanted.
\end{proof}

\begin{Rem}
Notice that neither of the two theorems of this section implies
the other. For example, for $a=1362$, the greatest prime $p$ which
is smaller than $a$ is $1361$, hence Theorem \ref{thm1main} shows
that $N(1362,b)=m^k$ implies $k\leq 2$ for $b>1362-1361+1=2$,
while Theorem \ref{thm2main} gives the result only for
$b>\sqrt{1362}/1.95\simeq 18.93$.

On the other hand, for $a=1360$, the greatest prime $p$ which is
smaller than $a$ is $1327$, hence Theorem \ref{thm1main} shows
that $N(1360,b)=m^k$ implies $k\leq 2$ for $b>1360-1327+1=34$,
while Theorem \ref{thm2main} gives the result for
$b>\sqrt{1360}/1.95\simeq 18.91$, which is a better bound.

In fact for $a$ big enough (say $a>5000$) the bound on $b$ of
Theorem \ref{thm1main} seems to be always better than the one of
Theorem \ref{thm2main}, as Figure 2 suggests.
\end{Rem}

\begin{figure}[h]
\includegraphics[scale=0.35]{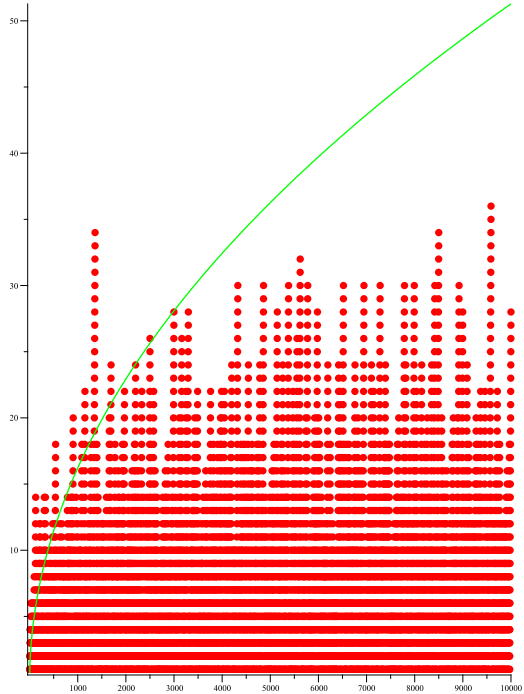}%[width=260mm,clip=true,trim=10mm 180mm 30mm 10mm]{Figure1.eps}
\caption{The red dots are the pairs $(a,a-p+1)$, where $p$ is the
greatest prime smaller than $a$, while the green curve is the
function $x\mapsto \sqrt{x}/1.95$.}
\end{figure}

\subsection{Conjecture \ref{conNara} for small values of $b$}

Notice that both theorems cover only cases in which $b$ is ``not
too small''. For small $b$'s things can become complicated, though
something can be said.

Consider for instance the case $b=2$. Then
\[N(a,2)=\frac{a(a-1)}{2}=\left(\!\!\!%
\begin{array}{c}
  a \\
  2 \\
\end{array}%
\!\!\!\right)\] is just a binomial.

The fact that $N(a,2)=m^k$ implies $k\leq 2$ is then proved in
\cite{gyory}.

\smallskip

Consider now the case $b=3$. Then the equation
$$
N(a,3)=\frac{a(a-1)^2(a-2)}{12}=m^k
$$
is equivalent to
$$
(a-1)^4-(a-1)^2=12m^k
$$
or
$$
\left(2(a-1)^2-1\right)^2-48m^k=1.
$$

Now it would follow from Pillai's generalization of Catalan's
conjecture that there are at most finitely many exceptions to Conjecture \ref{conNara} in this case:
\begin{con}[Pillai's conjecture] For any triple of positive
integers $a,b,c$, the equation $ax^n-by^m=c$ has only finitely
many solutions $(x,y,m,n)$ with $(m,n) \neq (2,2)$.
\end{con}
Pillai's conjecture is still open (it holds conditionally assuming the $abc$-conjecture).

In conclusion, other than some numerical evidence and the cases
covered in this work, Conjecture \ref{conNara} remains open.

\section*{Acknowledgements}

We thank Clemens Fuchs for bringing the reference \cite{gyory} to
our attention.

\end{document}